\documentclass[12pt,lettersize]{article}

\usepackage{amsmath,amssymb,amsthm}

\DeclareMathOperator{\spec}{spec}

\DeclareMathOperator{\Real}{Re}
\DeclareMathOperator{\Imag}{Im}
\DeclareMathOperator{\dist}{dist}

\renewcommand{\epsilon}{\varepsilon}

\newcommand{\union}{\bigcup}

\newcommand{\tw}{\widetilde}

\newcommand{\st}{\,|\,}
\newcommand{\maps}{\colon}

\newcommand{\poiss}[1]{\{#1\}}

\theoremstyle{plain}

\newtheorem{theorem}{Theorem}

\newtheorem{prop}[theorem]{Proposition}

\newtheorem{corollary}[theorem]{Corollary}
\newtheorem{lemma}[theorem]{Lemma}
\numberwithin{theorem}{section}
\numberwithin{equation}{section}

\theoremstyle{definition}

\theoremstyle{remark}

\newtheorem{remark}[theorem]{Remark}

\DeclareMathOperator{\aad}{ad}

\title{Diophantine Tori and Nonselfadjoint Inverse Spectral Problems}
\author{Michael A Hall\footnote{Partially supported by NSF grant EMSW21-RTG.}\\
Department of Mathematics\\
University of California\\
Los Angeles\\
CA 90095-1555, USA \\
\texttt{michaelhall@math.ucla.edu}}
\date{}

\begin{document}
\maketitle
 
\thanks

\begin{abstract}
We study a semiclassical inverse spectral problem based on a spectral asymptotics result of \cite{HSV}, which applies to small non-selfadjoint perturbations of selfadjoint $h$-pseudodifferential operators in dimension 2. The eigenvalues in a suitable complex window have an expansion in terms of a quantum Birkhoff normal form for the operator near several Lagrangian tori which are invariant under the classical dynamics and satisfy a Diophantine condition. In this work we prove that the normal form near a single Diophantine torus is uniquely determined by the associated eigenvalues. We also discuss the normalization procedure and symmetries of the quantum Birkhoff normal form near a Diophantine torus.
\end{abstract}

\section{Introduction} \label{intro}

Let $M$ denote either $\mathbf{R}^2$ or a compact, real analytic manifold of dimension 2, and let $\tw M$ denote a complexification of $M$, which is $\mathbf{C}^2$ in the Euclidean case, and a Grauert tube of $M$ in the compact analytic 2-manifold case. 

We study operators of the form $P_\epsilon = P + i\epsilon Q$, where $P,Q$ are analytic $h$-pseudodifferential operators on $M$ with principal symbols $p,q$, respectively, and $P$ is selfadjoint. The principal symbol of $P_\epsilon$ is then $p + i\epsilon q$, where $p$ is real. We will also make a non-degeneracy assumption on $\Real q$, but we do not require $Q$ to be selfadjoint. 

Consider a Lagrangian torus $\Lambda$, contained in an energy surface $p^{-1}(0) \cap T^*M$, which is invariant with respect to the Hamilton flow of $p$ and satisfies a Diophantine condition (see section \ref{dyn}). For simplicity, we will assume:
\begin{align}\label{integ}
\text{The Hamilton flow of $p$ is completely integrable in a neighborhood of $p^{-1}(0)$.}
\end{align}
This implies that in a neighborhood of $\Lambda$ the energy surface is foliated by invariant Lagrangian tori. 

In action-angle coordinates $(x,\xi)$ such that $\Lambda$ is the set $\{\xi = 0\} \subseteq T^* \mathbf{T}^2_x$, where $\mathbf{T}^2 = \mathbf{R}^2/2\pi \mathbf{Z}^2$, let 
\begin{align}\label{avg}
\langle q \rangle (\xi) = \frac{1}{(2\pi)^2} \int_{\mathbf{T}^2} q(x,\xi)\, dx
\end{align}
denote the spatial average of $\langle q \rangle$. Here we take $x$ to be a multi-valued function whose gradient is single-valued. We assume that $dp$ and $d\Real \langle q \rangle$ are linearly independent along $\Lambda$. Then, as is explained in section \ref{QBNF}, we may make a sequence of changes of variables which transforms the full symbol of $P_\epsilon$ into a Birkhoff normal form, which in this context means an asymptotic expansion in $(\xi,\epsilon,h)$ that is independent of $x$ to high order. Formally we may carry out this procedure on the level of operators by conjugating by a sequence of appropriately defined Fourier integral operators, obtaining what we call a quantum Birkhoff normal form for $P_\epsilon$. 

Under some further technical assumptions which we will explain later on, one of the main theorems of \cite{HSV} establishes, for any $\delta > 0$, asymptotics for the eigenvalues of $P_\epsilon$ in an $h$-dependent window in the complex plane
\begin{align}\label{box}
\left\{z \in \mathbf{C} \, \left| \, |\Real z| < \frac{h^\delta}{C},\ |\Imag z - \epsilon \Real F| < \frac{\epsilon h^\delta}{C}\right.\right\}.
\end{align}
Here $F = \langle q \rangle|_\Lambda = \langle q \rangle(0)$, and we assume that this average is not shared by any other invariant torus. The expansions are given in terms of a Bohr-Sommerfeld type condition and the quantum Birkhoff normal form of $P_\epsilon$ near $\Lambda$. 

Our goal in this work is to address the semiclassical inverse problem of determining the quantum Birkhoff normal form of $P_\epsilon$ from the eigenvalues in (\ref{box}), assuming the unperturbed operator $P$ is known. 

Inverse spectral problems have been studied for many years, as surveyed for example by Zelditch \cite{Zelditch}. Recently, \textit{semiclassical} inverse spectral problems have been investigated by several authors, such as Colin de Verdi{\`e}re \cite{CdVII}\cite{CdV}, Colin de Verdi{\`e}re-Guillemin \cite{CdVI}\cite{ColindeVerdiereGuillemin}, Guillemin-Paul-Uribe \cite{GuilleminPaulUribe}, Guillemin-Paul \cite{GuilleminPaul}, Guillemin-Uribe \cite{GuilleminUribe}, Hezari \cite{Hezari}, Iantchenko-Sj\"ostrand-Zworski \cite{IantchenkoSjostrandZworski}, V{\~u} Ng{\d{o}}c \cite{VuNgoc}. Often in inverse spectral problems one studies the wave trace, in the spirit of Guillemin \cite{Guillemin}. The non-selfadjoint case in dimension 2 is special because the eigenvalues may have an explicit description in terms of the Birkhoff normal form and Bohr-Sommerfeld type rules, an idea first explored in Melin-Sj\"ostrand \cite{MeSj}. In such a situation it seems most natural to recover the normal form directly from eigenvalue asymptotics. Our approach is taken very much in the spirit of Colin de Verdi{\`e}re \cite{CdV}. 

According to \cite{HSV}, the eigenvalues in the window (\ref{box}) form a distorted lattice, with horizontal spacing $\sim h$ and vertical spacing $\sim \epsilon h$. The window is of size $h^\delta$ by $\epsilon h^\delta$, for some $0 < \delta \ll 1$, which means that the asymptotic expansions are valid for a comparatively large number of eigenvalues (on the order of $h^{2(\delta-1)}$) as $h \to 0$. 

In addition, for the semiclassical inverse problem, we assume we know the eigenvalues for \textit{each} sufficiently small value of the semiclassical parameter (or possibly for a sequence of values of $h$ tending to 0). This provides a rich data set, from which we will recover information about the Birkhoff normal form using elementary order of magnitude arguments. In the perturbative, non-selfadjoint setting there is also the parameter $\epsilon$ to consider, since the results of \cite{HSV} apply to all values of $\epsilon$ such that $h^K \leq \epsilon \leq h^\delta$. We will exploit this flexibility to assume that we can choose $\epsilon$ so as to rule out any sort of degenerate relationship between $\epsilon$ and $h$. For simplicity, we shall assume more strongly that we know the eigenvalues for all values of $\epsilon$ in such a range. See also the remarks at the end of Section {\ref{Main}.

Our main result, stated informally, is the following (see Theorem \ref{thm1} for the precise statement)
\begin{theorem}
The eigenvalues of $P_\epsilon$ in (\ref{box}) determine the quantum Birkhoff normal form of $P_\epsilon$ near $\Lambda$. 
\end{theorem}

The plan of the paper is as follows. In Section \ref{assumptions}, we recall the setting and technical assumptions of \cite{HSV} needed to apply one of the main results of that paper. In Section \ref{QBNF}, we review the normal form construction in the present context, using some notation and methods due to S. V\~u Ng\d oc, and we also discuss symmetries and uniqueness of the normal form. In Section \ref{spectralasymptotics}, we recall a spectral asymptotics result of \cite{HSV} in a precise form, which is the basis of the inverse result. Finally, in Section \ref{Main}, we prove our main theorem.

{\bf Acknowledgements.} I would like to thank Michael Hitrik for suggesting the problem and for his patient guidance with this paper. Thanks also to San V\~ u Ng\d oc for some helpful comments, as well as allowing us to include some elements of an unpublished exposition of the normal form construction in Section \ref{BNF}. For a similar presentation in the case of the Birkhoff normal form near a nondegenerate minimum, see \cite{CN}. Thanks 
 as well the anonymous referees for suggesting several improvements to the exposition.

\section{Assumptions}\label{assumptions}
 
We will state our assumptions along the lines of section 7 of \cite{HSV}, in particular restricting our attention to the completely integrable case (\ref{integ}) rather than the most general case treated in that work. 



\subsection{Analyticity and general assumptions}

Let us assume that $P_\epsilon = P+i\epsilon Q$, with $\epsilon\in \mbox{neigh}(0,\mathbf{R})$, satisfies the same general assumptions as operators studied in \cite{HSV}, which we recall here for convenience.

When $M = \mathbf{R}^2$, assume that $P_\epsilon = P_\epsilon(x,hD_x;h) = P(x,hD_x;h) + i\epsilon Q(x,hD_x;h)$ is the Weyl quantization of a symbol which we also denote $P_\epsilon(x,\xi,\epsilon;h) = P(x,\xi;h) + i\epsilon Q(x,\xi;h)$. 
Assume that $P_\epsilon$ is a holomorphic function of $(x,\xi)$ in a complex tubular neighborhood of $\mathbf{R}^4\subseteq \mathbf{C}^4$. Assume that 
\begin{align}\label{orderfunction}
|P_\epsilon(x,\xi;h)| \leq \mathcal{O}(1)g(\Real(x,\xi))
\end{align}
in this neighborhood, where $g \geq 1$ is an order function in the sense that 
\begin{align*}{}
g(X) \leq C \left\langle X-Y \right\rangle^M g(Y),\quad C > 0,\ M > 0,\ X,Y\in \mathbf{R}^4. 
\end{align*}
Assume that $P$ and $Q$ have asymptotic expansions
\begin{align}
P(x,\xi;h) \sim \sum_{j=0}^\infty h^j p_j(x,\xi), \quad Q(x,\xi;h)\sim \sum_{j=0}^\infty h^j q_j(x,\xi)
\end{align}
valid in the space of holomorphic symbols satisfying the bound (\ref{orderfunction}). Let us also assume that the principal symbol $p = p_0$ satisfies an ellipticity condition at infinity,
\begin{align*}{}
|p(x,\xi)| \geq \frac{1}{C}g(\Real(x,\xi)),\quad |(x,\xi)| \geq C. 
\end{align*}

When $M$ is a compact, real analytic 2-manifold, assume that in any choice of local coordinates $P_\epsilon = P+i\epsilon Q$ is a differential operator of order $m$ with analytic coefficients, which themselves have asymptotic expansions in integer powers of $h$. Assume also that the principal symbol $p$ of $P$ satisfies an ellipticity condition near infinity,
\begin{align*}\label{}
|p(x,\xi)| \geq \frac{1}{C} \langle \xi \rangle^m, \quad (x,\xi)\in T^*M,\ |\xi| \geq C,
\end{align*} 
where we implicitly assume that $M$ has been equipped with an analytic Riemannian metric, so that the quantity $\langle \xi \rangle = \sqrt{1 + |\xi|^2}$ makes sense. Assume also that the underlying Hilbert space is $L^2(M,\mu(dx))$, where $\mu$ is the Riemannian volume form on $M$.

In both cases, we assume that $P$ is formally selfadjoint on $L^2$, which implies that $p$ is real. 

The above assumptions imply also that $P_\epsilon$ has a natural closed, densely defined realization on $L^2$, which has discrete spectrum in a fixed neighborhood of $0\in \mathbf{C}$ for $h,\epsilon$ small enough. Also, we have that $\spec(P_\epsilon)\cap \mbox{neigh}(0, \mathbf{C}) \subseteq \{z \st \Imag z = \mathcal{O}(\epsilon)\}$. 

\subsection{Assumptions on the classical dynamics}\label{dyn}


Assume the energy surface $p^{-1}(0) \cap T^* M$ is non-critical, i.e. $dp \neq 0$ along this set. For simplicity, we assume $p^{-1}(0) \cap T^* M$ is connected. Let 
\begin{align*}{}
H_p = \sum_{j = 1,2} \frac{\partial p}{\partial \xi_j} \frac{\partial}{\partial x_j} - \frac{\partial p}{\partial x_j}\frac{\partial}{\partial \xi_j}
\end{align*}
denote the Hamilton vector field of $p$ (in any choice of canonical coordinates). 


By the complete integrability assumption (\ref{integ}), there exists an analytic, real-valued function $\tw p$ such that $H_p \tw p = \poiss{p,\tw p} = 0$ with $d\tw p$ and $dp$ linearly independent almost everywhere. Here, $\poiss{\cdot,\cdot}$ denotes the Poisson bracket. Then the energy surface $p^{-1}(0)\cap T^*M$ decomposes as a disjoint union of compact, connected $H_p$-invariant sets, which we assume has the structure of a graph, in which edges correspond to families of regular invariant Lagrangian tori and vertices correspond to singular invariant sets.

Near an invariant torus $\Lambda$ we have real analytic action-angle coordinates $(x,\xi)$ such that $\Lambda = \{\xi = 0\}$ and $H_p|_\Lambda = a\cdot \partial_x$ for some frequency vector $a\in \mathbf{R}^2$. We refer to $\Lambda$ as a rational, irrational, or Diophantine torus if the vector $a$ has the corresponding property. Below, we will consider a Diophantine torus, i.e. one such that the frequencies $a$ satisfy
\begin{align}\label{diof}
|a\cdot k| \geq \frac{1}{C_0|k|^{N_0}}, \quad 0 \neq k \in \mathbf{Z}^n. 
\end{align}
for some $C_0>0,N_0>0$. (Here, $n = 2$.)


In action-angle coordinates near any such $\Lambda$, the principal symbol $p$ takes the form
\begin{align}\label{expansion}
p = p(\xi) = a\cdot \xi + \mathcal{O}(\xi^2).
\end{align}
In particular, it is independent of $x$, which means that it is in Birkhoff normal form. 

Let $q$ be the principal symbol of $Q(x,hD_x;h)$, and let $\left\langle q \right\rangle|_{\Lambda}$ denote the average as in (\ref{avg}) of $q$ with respect to the natural smooth measure on $\Lambda$. We assume that the analytic function $\Lambda\mapsto \Real \left\langle q \right\rangle|_\Lambda$ is not identically constant on any of the aforementioned ``edges'', consisting of families of invariant tori.


When $T > 0$, let $\left\langle q \right\rangle_T$ denote the symmetric time $T$ average of $q$ along the $H_p$-flow:
\begin{align*}{}
\left\langle q \right\rangle_T(x,\xi) = \frac{1}{T} \int_{-T/2}^{T/2} q\circ \exp(sH_p)(x,\xi)\, ds. 
\end{align*}
%
%
%
For each invariant torus $\Lambda$, define the interval
\begin{align}\label{Qinf}
Q_\infty(\Lambda) = \left[\lim_{T\to \infty} \inf_{\Lambda} \Real \langle q \rangle_T, \lim_{T\to \infty} \sup_{\Lambda}\Real \langle q \rangle_T \right].
\end{align}
As in \cite{SjDWE}, we have that $\spec P_\epsilon \cap \{|\Real z| \leq \delta\}$ is contained in a band 
\begin{align*}{}
\frac{\Imag z}{\epsilon} \in \left[ \inf \union_\Lambda Q_\infty(\Lambda) - o(1), \sup \union_\Lambda Q_\infty(\Lambda) + o(1)\right]
\end{align*}
as $\epsilon,h,\delta \to 0$. 

From now on, fix a single Diophantine invariant Lagrangian torus $\Lambda$, set $F = \left\langle q \right\rangle|_{\Lambda}$, and assume that 
\begin{align}\label{nondeg}
\text{ $dp$ and $d_\Lambda\Real\langle q \rangle|_\Lambda$ are linearly independent.}
\end{align}

With all the assumptions above, and in particular assuming complete integrability, the last global assumption needed is that
\begin{align}\label{global}
\Real F \not\in Q_\infty(\Lambda'), \quad \Lambda' \neq \Lambda. 
\end{align}
Without assuming complete integrability, a different assumption is needed (see \cite{HSV}, (1.24)).

The eigenvalue asymptotics result of \cite{HSV} is valid in the $(h,\epsilon)$-dependent rectangle (\ref{box}) for sufficiently small $h$ and assuming $h^K \leq \epsilon \leq h^\delta$, where $K$ is a fixed integer, which can be chosen arbitrarily large, and $\delta > 0$ is also fixed, and can be chosen arbitrarily small. 

\begin{remark}
Note that the global condition (\ref{global}) implies the value $F$ is unique to $\Lambda$. The results of \cite{HSV} apply also to a finite collection of Diophantine tori sharing the value $F$, in which case the set of eigenvalues in (\ref{box}) is simply the union of the contributions from each individual torus, modulo $\mathcal{O}(h^\infty)$. However, we will not consider the problem of separating the contributions of several tori. 
\end{remark}


\section{Quantum Birkhoff normal form}\label{QBNF}

In this section we present the quantum Birkhoff normal form construction near a Diophantine torus for a perturbed symbol, and discuss issues of uniqueness for the normal form and normalizing change of variables. 
Though we only need to consider dimension 2, it is natural to carry out the discussion in a general dimension $n$, as no changes are needed. We will work on $T^*\mathbf{T}^n$, assuming we are in a microlocal model where the Diophantine torus in question corresponds to the 0 section $\{\xi = 0\}$. 

\subsection{Normal form construction}\label{BNFconstr}

Let us identify symbols on $T^*\mathbf{T}^n$ with their formal Weyl quantizations. The Moyal formula
\begin{align*}
P\#^w Q(x,\xi,\epsilon;h) \sim \sum_{|\alpha|,|\beta| = 0 }^\infty \frac{h^{|\alpha| + |\beta|}(-1)^{|\alpha|}}{(2i)^{|\alpha|+|\beta|}\alpha!\beta!} (\partial_x^\alpha \partial_\xi^\beta P(x,\xi,\epsilon) ( \partial_\xi^\alpha\partial_x^\beta Q(x,\xi,\epsilon))
\end{align*}
defines a product operation on symbols which corresponds to composition of the corresponding operators. We denote by $[\cdot,\cdot]$ the associated bracket operation, which on the level of Weyl quantizations is simply the commutator bracket.

The normal form construction may be summarized as follows: We make a sequence of analytic, symplectic changes of variables, which transform $P_\epsilon$ to a symbol which is independent of $x$ to higher and higher order in $(\xi,\epsilon,h)$. On the level of operators this is formally equivalent to conjugating by a sequence of Fourier integral operators. The resulting sequence of symbols is convergent in the space of formal power series in $(\xi,\epsilon,h)$. The quantum Birkhoff normal form (QBNF) of $P_\epsilon$ near the Diophantine torus $\Lambda$ is {an asymptotic expansion which is} the formal limit of this procedure, while if we truncate the procedure after finitely many steps we get a well-defined analytic change of variables. 

Later, we will often write the QBNF as a formal expansion
\begin{align*}{}
  P^{(\infty)}(\xi,\epsilon,h) \sim \sum_{j,m,n} P^{(\infty)}_{jmn}(\xi)\epsilon^m h^n,
\end{align*}
where the sum is over integers $j\geq 0$, $m\geq 0$, $n\geq 0$, and $P_{jmn}$ is a homogeneous polynomial of degree $j$ in $\xi$, and $P^{(\infty)}$ denotes the entire formal expansion. 

For the moment, however, it is convenient to use slightly different notation. Consider a grading in $(\xi,\epsilon,h)$ which counts the power in $\xi$ plus {\it twice} the power in $(\epsilon,h)$. Let $\mathcal{O}(N)$ denote the associated order classes. Here we do not attach any special significance to the number two, but we note the convenience of this sort of grading: because the Moyal formula has an asymptotic expansion in powers of $(\frac{h}{i} \frac{\partial}{\partial \xi}, \frac{\partial}{\partial x})$, the higher weight of $h$ ensures that each time we lose an order in $\xi$ we gain one in $h$. This implies that the main contribution in the bracket $\frac{i}{h}[\cdot,\cdot]$ comes from the Poisson bracket of the two symbols. 

When $K_j = \mathcal{O}(j)$ and $K_\ell = \mathcal{O}(\ell)$, their Poisson bracket satisfies
\begin{align*}{}
\poiss{K_j,K_\ell} = \sum_{k} \frac{\partial K_j}{\partial \xi_k} \frac{\partial K_\ell}{\partial x_k}  - \frac{\partial K_j}{\partial x_k}\frac{\partial K_\ell}{\partial \xi_k} = \mathcal{O}(j + \ell - 1),
\end{align*}
By what we have said above, we see that $\frac{i}{h}[K_j,K_\ell] = \poiss{K_j,K_\ell} + \mathcal{O}(j + \ell)$. 


\begin{prop}\label{BNF}
Suppose  that $P = P_1 + \mathcal{O}(2)$ is analytic in $x$ and $\xi$, where $P_1 = a\cdot \xi$, and that $a$ satisfies the Diophantine condition (\ref{diof}). Then for all $N \geq 1$ there exist functions
\begin{align*}{}
G^{(1)} = 0,\quad G^{(N)} = G_2 + \cdots + G_N \ (N \geq 2), \quad P^{(N)} = P_1 + P_2 + \cdots + P_N, \quad  R_{N+1},
\end{align*}
which are analytic in $x$, with $G_j$, $P_j$ and $R_j$ homogeneous of degree $j$ with respect to the grading described above (thus polynomials in $\xi$), such that
\begin{align}\label{nth}
\textstyle  \exp\left( \frac{i}{h}\aad_{G^{(N)}}\right)P = P^{(N)} + R_{N+1} + \mathcal{O}(N+2)
\end{align}
with each $P_j$ independent of $x$. 
\end{prop}

Here, we write formally $\aad_G P = [G,P]$ and $\exp(\frac{i}{h}\aad_G) P = \exp(\frac{i}{h}G)P\exp(-\frac{i}{h}G)$, but note that these do not represent concretely defined operators. Notice also that $G^{(N)} = \mathcal{O}(2)$ for all $N \geq 1$. 

\begin{remark}
Note that although it is not ruled out by the notation, it will follow from the proof that no half-powers of $h$ or $\epsilon$ appear in the normal form. 
\end{remark}

\begin{proof}
We proceed by induction on the order $N$. By assumption the claim holds for $N = 1$, with $G_1 = 0$, and $R_2$ representing the homogeneous terms of degree 2. 
Inductively if (\ref{nth}) holds, then setting $G^{(N+1)} = G^{(N)} + G_{N+1}$, for some function $G_{N+1}$, homogeneous of degree $N+1$, to be determined, we claim that the only new term modulo $\mathcal{O}(N+2)$ is given by $\poiss{G_{N+1},P_1}$. Indeed, by the Campbell-Hausdorff formula
\begin{align*}{}
\textstyle \exp\left( \frac{i}{h}\aad_{G^{(N+1)}}\right)P
&=\textstyle   \exp\left( \frac{i}{h}\aad_{G_{N+1}} + \frac{i}{h} \aad_{G^{(N)}}\right) P \\
&=\textstyle  \exp\left( \frac{i}{h}\aad_S\right)\circ\exp\left( \frac{i}{h}\aad_{G_{N+1}}\right) \circ \exp\left( \frac{i}{h}\aad_{G^{(N)}}\right) P
\end{align*}
where $S = \mathcal{O}(\frac{1}{2} [G_{N+1},G^{(N)}]) = \mathcal{O}((N+1)+2-1) = \mathcal{O}(N+2)$. Here we have used that $G^{(N)}= \mathcal{O}(2)$.  Meanwhile 
\begin{align*}{}
\textstyle \exp\left( \frac{i}{h}\aad_{G_{N+1}}\right) \circ \exp\left( \frac{i}{h}\aad_{G^{(N)}}\right) P
&= \textstyle  \exp\left( \frac{i}{h}\aad_{G_{N+1}}\right)(P^{(N)} + R_{N+1}) = \mathcal{O}(1).
\end{align*}
Because of the Campbell-Hausdorff formula and the order of $S$, applying $\exp(\frac{i}{h}\aad_S)$ only affects the terms of order $\mathcal{O}(1+(N+2)-1) =  \mathcal{O}(N+2)$. Therefore we have
\begin{align*}{}
\textstyle \exp\left( \frac{i}{h}\aad_{G^{(N+1)}}\right)P
&= \textstyle \exp\left( \frac{i}{h}\aad_{G_{N+1}}\right)(P^{(N)} + R_{N+1}) + \mathcal{O}(N+2)\\
&= \textstyle P^{(N)} + R_{N+1} + \frac{i}{h}\aad_{G_{N+1}} (P^{(N)} + R_{N+1}) + \mathcal{O}(N+2) \\
&= \textstyle P^{(N)} + R_{N+1} + \frac{i}{h}\aad_{G_{N+1}} (P_1+\mathcal{O}(2)) + \mathcal{O}(N+2) \\
&= \textstyle P^{(N)} + R_{N+1} + \frac{i}{h}\aad_{G_{N+1}} P_1 + \mathcal{O}(N+2),
\end{align*}
Because the bracket $\frac{i}{h}[\cdot,\cdot]$ reduces to the Poisson bracket when one of the arguments is at most quadratic, we have $\frac{i}{h}\aad_{G_{N+1}}P_1 = \frac{i}{h}[G_{N+1},P_1] = \poiss{G_{N+1},P_1}$, so 
\begin{align*}{}
\textstyle \exp\left( \frac{i}{h}\aad_{G^{(N+1)}}\right)P = P^{(N)} + R_{N+1} + \poiss{G_{N+1},P_1} + \mathcal{O}(N+2),
\end{align*}
as claimed.

To make the homogeneous order $N+1$ terms independent of $x$, it suffices to solve the cohomological equation,
\begin{align}\label{cohom}
\poiss{ G_{N+1},P_1} = \left\langle R_{N+1} \right\rangle - R_{N+1},
\end{align}
for $G_{N+1}$, where $\langle R_{N+1} \rangle$ is the $x$-average of $R_{N+1}$ as in (\ref{avg}). Indeed, assuming that we have done so, we then set $P_{N+1} = \left\langle R_{N+1} \right\rangle$ and let $R_{N+2}$ represent the homogenous order $N+2$ part of the $\mathcal{O}(N+2)$ error terms, which are analytic.  

To solve (\ref{cohom}), note that because $P_1 = a\cdot \xi$, we have 
\begin{align*}
\poiss{G_{N+1},P_1} = -H_{P_1} G_{N+1} = -(a\cdot \partial_x) G_{N+1}.
\end{align*}
Expanding $G_{N+1}$ and $R_{N+1}$ in Fourier series
\begin{align*}{}
G_{N+1}(x,\xi) = \sum_{k\in \mathbf{Z}^n} \widehat{G}_{N+1}(k,\xi)e^{i k\cdot x} \\
R_{N+1}(x,\xi) = \sum_{k\in \mathbf{Z}^n} \widehat{R}_{N+1}(k,\xi)e^{i k\cdot x},
\end{align*}
we have
\begin{align*}{}
(a\cdot \partial_x) G_{N+1} = \sum_{k\in \mathbf{Z}^n} (i a\cdot k)\widehat{G}_{N+1}(k,\xi)e^{i k\cdot x}.
\end{align*}
When $0 \neq k \in \mathbf{Z}^n$, because $a\cdot k$ does not vanish by the Diophantine condition (\ref{diof}), we may set $\widehat{G}_{N+1}(k,\xi) = -i \widehat{R}_{N+1}(k,\xi) / (a\cdot k)$, to obtain $(a\cdot \partial_x) G_{N+1} = R_{N+1} - \widehat{R}_{N+1}(0,\xi) = R_{N+1} - \left\langle R_{N+1} \right\rangle$ and thus solve the cohomological equation. Furthermore, again by (\ref{diof}), we have
\begin{align*}{}
|\widehat{G}_{N+1}(k,\xi)| = \frac{|\widehat{R}_{N+1}(k,\xi)|}{i |a\cdot k|} \leq C_0 |k|^{N_0}|\widehat{R}_{N+1}(k,\xi)|,
\end{align*}
and because $R_{N+1}$ is analytic in a neighborhood of $\mathbf{T}^n \times \{0\} \subseteq (\mathbf{C}^{n}/ 2\pi \mathbf{Z}^n) \times\mathbf{C}^n$, so too is $G_{N+1}$.
\end{proof}


\begin{remark}
As mentioned above, $\exp(\frac{i}{h}\aad_{G^{(N)}})$ formally represents conjugation by a microlocally defined Fourier integral operator $\exp(\frac{i}{h}G^{(N)})$, and such conjugation microlocally implements the symplectic transformation $\exp H_{G^{(N)}}$. To define such operators concretely, so that they act on microlocally defined distributions, one works on the FBI transform side in suitable weighted spaces of holomorphic functions. An additional assumption about smoothness in $\epsilon$ is required. See \cite{HSV} and the references there. 
\end{remark}


\subsection{Symmetries and uniqueness of the normal form}\label{BNFuniq}

In this section we discuss symmetries and uniqueness of the Birkhoff normal form.

We first remark that there is always some flexibility in choosing action-angle variables $(x,\xi)$ near an invariant torus (here $x \in \mathbf{T}^n$ represents the angle variables, and $\xi \in \mathbf{R}^n$ the action variables). If $A\in \mbox{GL}\,(n,\mathbf{Z})$ and $\psi$ is any smooth function on $\mathbf{R}^n$, then 
\begin{align}\label{changewithA}
\kappa:(y,\eta) \mapsto (x,\xi) = (A^{-1}y + \partial \psi(\eta), A^t \eta) 
\end{align}
gives a well-defined smooth, symplectic change of variables (which is analytic if $\psi$ is analytic) on $\mathbf{T}^n \times \mathbf{R}^n \cong T^*\mathbf{T}^n$, and thus a new set of action-angle coordinates $(y,\eta)$. 

This transformation also preserves independence of the angle coordinate, and thus takes one asymptotic expansion which is in normal form (to order $N$) to another. More precisely, if 
\begin{align*}{}
p(x,\xi,\epsilon,h) = \tw p(\xi,\epsilon,h) + \mathcal{O}(N) = a\cdot \xi + \mathcal{O}(2),
\end{align*}
then in the new coordinates
\begin{align*}{}
(p\circ \kappa)(y,\eta,\epsilon,h) = \tw p(A^t\eta,\epsilon,h) + \mathcal{O}(N) = (Aa)\cdot \eta + \mathcal{O}(2).
\end{align*}


Once we fix a choice of frequencies $a$, then, $A$ must be the identity (because of the Diophantine assumption), which means we only have maps of the form
\begin{align}\label{change}
(y,\eta) \mapsto (y +\partial\psi, \eta),
\end{align}
which do not affect a normal form expansion because the second coordinate is unchanged.

Our aim is to show that the formula (\ref{change}) gives all transformations which preserve independence of the angle variables in a general function or asymptotic expansion, while not affecting the frequencies $a$. Appendix A.1 of \cite{HoferZehnder} essentially contains a proof using generating functions that if a real symplectic diffeomorphism $(y,\eta)\mapsto (x,\xi)$ satisfies $\xi = b(\eta)$ with $\det(b_\eta) \neq 0$, then the mapping is of the form (\ref{changewithA}). The only difference is that the argument given there is local, and they end up with a more general type of transformation. In our case, it turns out one can make the formula apply globally, and then because we are on a torus, periodicity forces the simpler form (\ref{changewithA}), which reduces to (\ref{change}) assuming $a$ is unchanged. This argument will be given in Proposition \ref{uniq} below. 

Before proceeding, however, we note that because we have used complex symplectic transformations in our reduction to the normal form, it is natural to also consider symplectic biholomorphisms in a small complex neighborhood of $\mathbf{T}^n \times \{0\}$, for example allowing $\psi$ to be complex-valued in (\ref{change}). Here when we say a holomorphic transformation is symplectic or canonical we mean that the mapping preserves the standard symplectic form $\sigma$ on $\mathbf{C}^n / 2\pi \mathbf{Z}^n \times \mathbf{C}^n$, which is a form of type $(2,0)$, given in coordinates by 
\begin{align*}
\sigma = \sum_{j=1}^n d\xi_j \wedge dx_j.
\end{align*}

The standard fact that symplectomorphisms admit local generating functions carries over to the complex setting: 
\begin{prop}\label{genfn}
If $\kappa : (y,\eta)\mapsto (x,\xi) = (a(y,\eta),b(y,\eta))$ is a symplectic biholomorphism between small neighborhoods of $(y_0,\eta_0)\in \mathbf{C}^n$ and $(x_0,\xi_0)\in \mathbf{C}^n$, and $\det(b_\eta)\neq 0$, then there exists a holomorphic function $\phi(y,\xi)$ such that 
\begin{align*}
\kappa: (y,\partial_y \phi) \mapsto (\partial_\xi \phi, \xi). 
\end{align*}
\end{prop}

The proof is exactly the same as in the real case (see appendix A.1 of \cite{HoferZehnder} for example), if we note that the implicit function theorem holds for holomorphic maps, and where normally we use Poincar\'e's lemma we instead use the Dolbeault-Grothendieck lemma.

Using this fact, we may now prove

\begin{prop} \label{uniq}
Suppose $\kappa\maps (y,\eta) \mapsto (x,\xi)$ is a symplectic biholomorphism between open sets $U,V \subseteq (\mathbf{C}^n / 2\pi\mathbf{Z}^n) \times \mathbf{C}^n$, where $U$ and $V$ are small neighborhoods of $(\mathbf{R}^n/2\pi\mathbf{Z}^n) \times \{0\}$. Suppose $\xi = b(\eta)$, where $b$ is a biholomorphism defined near $0 \in \mathbf{C}^n$, such that $b(0) = 0$, $b_\eta(0) = 1$. Then in a small enough neighborhood of $(\mathbf{R}^n/2\pi\mathbf{Z}^n) \times \{0\}$, $\kappa$ satisfies 
\begin{align*}
\kappa: (y,\eta) \mapsto (y +  \partial \psi(\eta), \eta).
\end{align*}
for some analytic function $\psi$ defined in a neighborhood of $0 \in \mathbf{C}^n$.
\end{prop}

\begin{proof}

We may lift $\kappa$ to a mapping between small neighborhoods of $\mathbf{R}^n \times \{0\} \subseteq \mathbf{C}^{2n}$. We use the same notations for the lift. By continuity, in small enough neighborhoods we have that $\det(b_\eta) \neq 0$, so by Proposition \ref{genfn}, locally we may find an analytic generating function $\phi(y,\xi)$ so that the mapping is given by
\begin{align}\label{gen}
(y,\partial_y\phi)\mapsto (\partial_\xi \phi,\xi).
\end{align}
When two neighborhoods overlap, the local generating functions $\phi,\tw \phi$ must agree modulo constants, and in this way we may define a globally defined analytic generating function $\phi(y,\xi)$, so that (\ref{gen}) is satisfied everywhere. Moreover the derivatives of $\phi$ are $2\pi\mathbf{Z}^n$-periodic, so (\ref{gen}) descends to $U,V$. 

Still working with the lift, if $\partial_y \phi = \beta(\xi)$ is the inverse of the diffeomorphism $b$, then for some function $\tw \psi$, we have
\begin{align*}{}
\phi(y,\xi) &= \beta(\xi)\cdot y + \tw \psi(\xi) \\
\partial_\xi \phi &= \partial_\xi \beta\cdot y + \partial_\xi \tw \psi. 
\end{align*}
Thus we have
\begin{align*}{}
(y,\beta(\xi)) \mapsto (\partial_\xi \beta\cdot y + \partial_\xi \tw \psi,\xi).
\end{align*}
Because the mapping is $2\pi\mathbf{Z}^n$-periodic in $x$ and $y$, we must have that for each $\xi_0 \in \mathbf{C}^n$, the assignment $\mathbf{Z}^n \ni n  \mapsto \partial_\xi \beta(\xi_0) \cdot n$ defines an automorphism of $\mathbf{Z}^n$. By continuity, $\partial_\xi\beta \in\mbox{GL}\,(n,\mathbf{Z})$ must be constant, and since $\beta(0) = 0$, with $\partial_\xi \beta(0) = 1$, $\beta$ itself must be the identity mapping. 

Therefore, setting $\psi = \tw \psi \circ b(\eta)$, the original mapping is given by
\begin{align*}
(y,\eta) \mapsto (y +\partial \psi(\eta), \eta),
\end{align*}
with generating function $\phi(y,\xi) = \xi\cdot y + \tw \psi(\xi)$, and $\tw \psi$ is now allowed to be complex-valued. 

\end{proof}

\begin{remark}
Note that Proposition \ref{uniq} only classifies transformations which preserve independence of the angular variables for a {\it general} symbol $\tw p(\xi,\epsilon,h) = a\cdot \xi + \mathcal{O}(2)$ (while leaving the frequencies $a$ unchanged). That is, if a transformation has this property with respect to {\it any} such $\tw p$, then it is easy to see the transformation must satisfy the hypotheses of Proposition \ref{uniq}. A priori, some specific symbol may admit more symmetries than the type described. See also Corollary \ref{coruniq}.
\end{remark}

\section{Eigenvalue asymptotics}\label{spectralasymptotics}
 
Under the assumptions stated in Section \ref{assumptions}, Theorem 1.1 in \cite{HSV} implies that if $h^K \leq \epsilon\leq h^\delta$, the eigenvalues of $P_\epsilon = P + i\epsilon Q$ which lie in (\ref{box}) have an asymptotic expansion given in terms of the QBNF of $P_\epsilon$. 

\begin{theorem}[\cite{HSV}]\label{hsvthm}
  Let $P_\epsilon$ satisfy all the assumptions of Section \ref{assumptions}. For each $N \geq 1$, let $P^{(N)}$ be the result of applying Proposition \ref{BNF} to $P_\epsilon$, and let us write
\begin{align}\label{spec}
P^{(N)} = P^{(N)}(\xi,\epsilon,h) = \sum_{j+2(m+n) \leq N} P^{(\infty)}_{jmn}(\xi)\epsilon^m h^n,
\end{align}
where $P^{(\infty)}_{jmn}$ is a homogeneous polynomial of degree $j$ in $\xi$ which does not depend on $N$.

For any $0 < \delta < 1$, and any fixed integer $K$, suppose that $h^K \leq \epsilon \leq h^\delta$. Recall the complex window (\ref{box}), where we now take $C > 0$ to be sufficiently large. Then for each $N$, as $h\to 0$, the quasi-eigenvalues
\begin{align}\label{quasi}
\textstyle  P^{(N)}\left(h\left(k-\frac{k_0}{4}\right) - \frac{S}{2\pi},\epsilon,h\right),\quad k \in \mathbf{Z}^2
\end{align}
are equal to the eigenvalues of $P_\epsilon$ modulo $\mathcal{O}(N+1)$ in (\ref{box}), in the sense that for all sufficiently small $h$, there is a one-to-one partial function from the set of quasi-eigenvalues to $\spec(P_\epsilon)$, equal to $1 + \mathcal{O}(N+1)$ uniformly, which is defined whenever a quasi-eigenvalue or the targeted true eigenvalue lies in (\ref{box}). 
\end{theorem}

The formulation is slightly different from Theorem 1.1 in \cite{HSV} because of the different grading we have chosen, but the same proof applies. See also Theorems 5.1 and 5.2 of \cite{HSV} for somewhat more direct analogues to the above with the alternate grading.

The expression which appears in place of $\xi$ in (\ref{quasi}) is the result of a Bohr-Sommerfeld type condition. The constant vector $k_0\in \mathbf{Z}^2$ contains the Maslov indices and $S\in \mathbf{R}^2$ the actions along a set of fundamental cycles of $\Lambda$, for example $\{x_1 = 0\}$, $\{x_2 = 0\}$, with respect to action-angle variables chosen so that $\Lambda$ is represented as $\{\xi = 0\} \subseteq T^*\mathbf{T}^2$. For more details on this point, see \cite{HSV}, as well as section 2 of \cite{HS}.

\section{Main Result}\label{Main}
 
Our main result is a uniqueness statement asserting that if the eigenvalues corresponding to invariant torus $\Lambda$ are the same for operators $P + i\epsilon Q_1,P+i\epsilon Q_2$, then they have the same QBNF near $\Lambda$.

\begin{theorem}\label{thm1}
Suppose that $P_1 = P+i\epsilon Q_1$, $P_2 = P+i\epsilon Q_2$ are operators satisfying the assumptions  described in section \ref{assumptions}, where $P = P(x,hD_x;h)$ is a fixed, self-adjoint operator with principal symbol $p$, $\Lambda$ is an $H_p$-invariant Lagrangian torus satisfying the Diophantine condition (\ref{diof}), and $Q_1$, $Q_2$ are operators with principal symbols $q_1,q_2$, respectively, such that $\left\langle q_1 \right\rangle|_{\Lambda} = \left\langle q_2 \right\rangle|_{\Lambda}$, and $\Lambda$ satisfies the global condition (\ref{global}) with respect to $P + i\epsilon Q_\nu$, $\nu = 1,2$.

Fix $0 < \delta < 1$, and let the $(\epsilon,h)$-dependent rectangle $R_\delta \subseteq \mathbf{C}$ be as described in equation (\ref{box}), 
\begin{align}\label{Rdelta}
  R_\delta = \left \{z \in \mathbf{C}\,\left|\, |\Real z| < \frac{h^\delta}{C},\ |\Imag z - \epsilon \Real F| < \frac{\epsilon h^\delta}{C}\right.\right \},
\end{align}
where $C > 0$ is large enough, with $F = \left\langle q_1 \right\rangle|_{\Lambda} = \left\langle q_2 \right\rangle|_{\Lambda}$ . For $\nu = 1,2$, let $P^{(\infty)}_\nu$, denote the QBNF of $P + i\epsilon Q_\nu$, which is a formal asymptotic expansion of the form 
\begin{align*}{}
P^{(\infty)}_\nu(\xi,\epsilon,h) \sim \sum_{j,m,n} P^{(\infty)}_{\nu,jmn}(\xi) \epsilon^m h^n = a\cdot \xi + i\epsilon F + p_{01}h + \mathcal{O}((|\xi|,\epsilon,h)^2), 
\end{align*}
%
where $P^{(\infty)}_{\nu,jmn}$, $j\geq 0$, $m \geq 0$, $n\geq 0$, is a homogeneous polynomial in $\xi$ of degree $j$. Suppose that for all sufficiently small $h$ and for all $\epsilon$ in the range $h^K \leq \epsilon \leq h^\gamma$, where $K \geq 1$ and $\gamma > \frac{1}{2}$ are fixed, we have
\begin{align*}{}
\spec(P+i\epsilon Q_1)\cap R_\delta = \spec(P+i\epsilon Q_2)\cap R_\delta. 
\end{align*}
Then the QBNF's $P^{(\infty)}_1$, $P^{(\infty)}_2$ are equal, i.e., for all $j,m,n$ we have $P^{(\infty)}_{1,jmn} = P^{(\infty)}_{2,jmn}$.
\end{theorem}

Whenever $k \in \mathbf{Z}^2$, in what follows we will use the notation $\xi_k = h(k-\frac{k_0}{4})- \frac{S}{2\pi}$. To simplify the notation further, let us write $P^{(\infty)}_\nu(k)$ for the eigenvalue with asymptotic expansion given by $P^{(\infty)}_\nu(\xi_k,\epsilon,h)$, according to Theorem \ref{hsvthm}.

We first prove a lemma which says that for an operator $P_\epsilon$, we can recover the Bohr-Sommerfeld index $k$ from the associated eigenvalue $P^{(\infty)}_\epsilon(k)$ when, as in the hypotheses of the theorem, we have a bound $\epsilon\lesssim  h^\gamma$ for some $\gamma> \frac{1}{2}$. We remark that the expononent $\frac{1}{2}$ is not fundamental, but essentially comes out of the proof.

\begin{lemma}\label{lemma}
Let $P^{(\infty)}_1,P^{(\infty)}_2$ represent two QBNF's arising from operators which satisfy the hypotheses of Theorem \ref{thm1}. In particular, assume 
\begin{align}\label{pstart}
P^{(\infty)}_\nu(\xi,\epsilon;h) = a\cdot \xi + i\epsilon F + p_{01}h + \mathcal{O}((|\xi|,\epsilon,h)^2), \quad \nu = 1,2,
\end{align}
where $p_{01}$ is a real constant and the vector $a$ satisfies the Diophantine condition
\begin{align*}{}
|a\cdot k| \geq \frac{1}{C_0|k|^{N_0}}, \qquad 0 \neq k \in \mathbf{Z}^2.
\end{align*}
Then for any $\gamma > \frac{1}{2}$, if $\epsilon \lesssim  h^\gamma$, there exists $\beta_0 < 1$ and $h_0 > 0$ such that for all $\beta \in [\beta_0,1)$, when $h\in [0,h_0)$, if $k,\ell \in \mathbf{Z}^2$ satisfy
\begin{align*}
P^{(\infty)}_1(k) = P^{(\infty)}_2(\ell) \mod \mathcal{O}(h^\infty),
\end{align*}
with $P^{(\infty)}_1(k),P^{(\infty)}_2(\ell)$ lying in the window $R_\beta$, then $k = \ell$. 
\end{lemma}

\begin{remark}
Note that we want $0 < \beta < 1$ in order to have many eigenvalues $P^{(\infty)}_\nu(k)$ which lie in the rectangle $R_\beta$. Indeed, the dimensions of $R_\beta$ are $\sim h^\beta\times \epsilon h^\beta$, so by the nondegeneracy assumption (\ref{nondeg}),  $\#\spec(P_\epsilon)\cap R_\beta \sim h^{2\beta-2}$ and if $P^{(\infty)}_\nu(k)\in R_\beta$ then $|\xi_k| = \mathcal{O}(h^\beta)$. 
\end{remark}

\begin{proof} By the hypotheses on $P^{(\infty)}_1,P^{(\infty)}_2$, and in particular because the the two expansions have several terms in common, as indicated in (\ref{pstart}), we have
\begin{align}\label{difference}
P^{(\infty)}_1(k)-P^{(\infty)}_2(\ell) = ha\cdot(\xi_k-\xi_\ell) + \mathcal{O}((|\xi_k|,\epsilon,h)^2) + \mathcal{O}((|\xi_\ell|,\epsilon,h)^2). 
\end{align}
By the Diophantine condition on $a$, we have 
\begin{align*}
|a\cdot(\xi_k-\xi_\ell)| = h|a\cdot(k-\ell)| \gtrsim \frac{h}{|k-\ell|^{N_0}}, \quad k\neq \ell. 
\end{align*}
Let us consider eigenvalues $P^{(\infty)}_\nu(k)\in R_\beta$ with $k\in \mathbf{Z}^2$, for some $\beta < 1$ to be determined. For any $k,\ell$ such that $|\xi_k|,|\xi_\ell|\lesssim h^\beta$, we have $h|k-\ell| = |\xi_k-\xi_\ell| \lesssim h^\beta$, so $|k-\ell| \lesssim h^{\beta-1}$. Hence, 
\begin{align}\label{hence}
h|a\cdot(k-\ell)| \gtrsim  \frac{h}{|k-\ell|^{N_0}} \gtrsim \frac{h}{(h^{\beta-1})^{N_0}} = h^{1+{N_0}(1-\beta)}, \quad k\neq \ell. 
\end{align}
The lemma will follow if we can show that by choosing $\beta$ close enough to 1, the above dominates the contributions of the error terms in (\ref{difference}). 

Thus, assuming $|\xi_k| \lesssim h^\beta$ and $\epsilon\leq  h^\gamma$, 
we estimate a typical term of one of the QBNF's:
\begin{align*}{}
|P^{(\infty)}_{\nu,jmn}(\xi)\epsilon^mh^n| \lesssim |\xi|^j \epsilon^m h^n| \lesssim h^{j\beta + m \gamma + n}. 
\end{align*}
To have $P^{(\infty)}_{\nu,jmn} = o(a\cdot(k-\ell))$, in view of (\ref{hence}), it suffices to have $h^{j\beta + m\gamma + n} = o(h^{1+{N_0}(1-\beta)})$, or
\begin{align*}{}
1+{N_0}(1-\beta) < j\beta + m\gamma + n.
\end{align*}
After rearranging, this is equivalent to 
\begin{align}\label{betacondition}
\frac{{N_0}+1-n-m\gamma}{{N_0}+j} < \beta. 
\end{align}
Examining the left hand side, we see that it is strictly less than 1, except in the following cases:
\begin{itemize}
\item[(a)] $j = 1$ and $m = n = 0$
\item[(b)] $j = m = 0$, and $n = 1$
\item[(c)] $j = n = 0$ and $m\gamma < 1$. 
\end{itemize}
However, the terms corresponding to these cases are exactly those written out in the right hand side of (\ref{pstart}). Cases (a) and (b) correspond to the terms $a\cdot \xi$ and $p_{01}h$, respectively, while our hypothesis $\gamma > \frac 12$ is designed so that case (c) only applies when $m = 1$, which corresponds to the term $i\epsilon F$. 

It follows that the exceptional cases above occur when $j + m + n = 1$. The left hand side of (\ref{betacondition}) must be maximized in one of the cases where $j + m + n = 2$, as increasing any of $j,m,n$ only makes the left hand side of (\ref{betacondition}) smaller. Therefore, if we choose $\beta$ so that
\begin{align*}
\frac{{N_0}+1-n-m\gamma}{{N_0}+j} < \beta < 1
\end{align*}
in each of these finitely many cases (which we will not list), then these inequalities also hold for all $j,m,n$ with $j + m + n \geq 2$. 

Therefore, for some large $M$, we have that for sufficiently small $h$, 
\begin{align*}
|P^{(\infty)}_1(k)&-P^{(\infty)}_2(\ell)| \\
&\geq h|a\cdot(k-\ell)| - \sum_{2 \leq j + m + n \leq M} |(P^{(\infty)}_{1,jmn}(\xi_k)-P^{(\infty}_{2,jmn}(\xi_\ell)) \epsilon^m h^n| - \mathcal{O}((|\xi_k|,|\xi_\ell|,\epsilon,h)^M) \\
&\gtrsim h^{1+{N_0}(1-\beta)} - \mathcal{O}(1) \sum_{2 \leq j+ m + n \leq M} h^{j\beta + m\gamma +n} - \mathcal{O}(h^{\frac{M}{2}}) \\
&\gtrsim h^{1+N_0(1-\beta)} - \mathcal{O}(h^{\frac{M}{2}}) \\
&\gtrsim h^{1+N_0(1-\beta)},
\end{align*}
where the last estimate holds when $M$ is chosen sufficiently large. The bound $\mathcal{O}(h^\frac{M}{2})$ comes from the fact that $\epsilon \lesssim h^\gamma$ with $\gamma > \frac{1}{2}$ and $|\xi_k|,|\xi_\ell| \lesssim h^\beta$, where $\beta$ is close to 1 (and so greater than $\frac{1}{2}$, we may assume).  Note that the implicit constants depend only on the terms of the QBNFs and not on $k,\ell$, so we get a uniform lower bound on the size of $P^{(\infty)}_1(k)-P^{(\infty)}_2(\ell)$ when $k \neq \ell$ and the eigenvalues lie in $R_\beta$. 

Summing up, we have found that there exists $\beta < 1$ such that if $P^{(\infty)}_1(k),P^{(\infty)}_2(\ell) \in R_\beta$, and $k\neq \ell$, then when $h$ is sufficiently small,
\begin{align*}
|P^{(\infty)}_1(k)-P^{(\infty)}_2(\ell)| \gtrsim h^{1+N_0(1-\beta)}
\end{align*}
where the implicit constant does not depend on $k,\ell$. Therefore, when $h$ is smaller than some small constant, which does not depend on $k,\ell$, if $P^{(\infty)}_1(k) = P^{(\infty)}_2(\ell)$, we must have $k = \ell$. 

\end{proof}

We now proceed to the proof of the main theorem. 

\begin{proof}[Proof of Theorem \ref{thm1}]
Theorem \ref{hsvthm} applies to $P_1,P_2$, meaning that their eigenvalues in a rectangle $R_\delta$ as in (\ref{Rdelta}) have asymptotic expansions of the form (\ref{quasi}), in terms of the QBNF's $P^{(\infty)}_1,P^{(\infty)}_2$. We will show one can recover the QBNF from the eigenvalues in any window $R_{\delta'}$ where $\beta_0 \leq \delta' < 1$, with $\beta_0$ chosen as in Lemma \ref{lemma}. Note that $\delta' \leq \delta$ implies $R_{\delta} \subseteq R_{\delta'}$, so we may assume without loss of generality that $\delta' = \delta$, and simply refer to both as $\delta$.

Suppose that $P^{(\infty)}_{1,jmn} \neq P^{(\infty)}_{2,jmn}$ for some index $(j,m,n)$. Then $P^{(\infty)}_{1,jmn} - P^{(\infty)}_{2,jmn}$ is a homogeneous polynomial of degree $j$ in $\xi$ which does not vanish identically. By homogeneity, we can find an open subset of the unit circle $U_{jmn} \subseteq \{|\xi| = 1\} \subseteq \mathbf{R}^2$ on which this polynomial is bounded away from 0 in absolute value. For all $h$ sufficiently small, there exists $\xi_k = h(k-\frac{k_0}{4}) - \frac{S}{2\pi}$ such that 
\begin{align}\label{xik}
|\xi_k| \sim h^\beta\quad \text{with}\quad \xi_k/|\xi_k| \in U_{jmn}.
\end{align}
For such $\xi_k$, by homogeneity we have
\begin{align*}{}
|P^{(\infty)}_{1,jmn}(\xi_k)\epsilon^mh^n-P^{(\infty)}_{2,jmn}(\xi_k)\epsilon^mh^n| \sim \epsilon^m h^{j\beta + n}.
\end{align*}
Let us also take $\epsilon = h^\gamma$, so that $\epsilon^m h^{j\beta +n} \sim h^{j \beta +m\gamma + n}$. 

Without loss of generality we may assume that, possibly after increasing $\beta$ and $\gamma$ slightly, we have that $j \beta +m\gamma+ n = j'\beta + m'\gamma + n'$ implies $m=m'$, $j=j'$, $n=n'$ (we just need $1,\beta,\gamma$ to be independent over the rationals). Then we have a total ordering of indices $(j,m,n)$ according to the size of the expression $h^{j\beta + m\gamma + n}$. Also, because $\beta, \gamma > 0$, for any fixed $M$ > 0, there are only finitely many indicies $(j,m,n)$ such that $j\beta + m\gamma + n \leq M$. Therefore, of the indices $(j,m,n)$ for which  $P^{(\infty)}_{1,jmn} \neq P^{(\infty)}_{2,jmn}$, there is a unique index for which $j\beta + m\gamma + n$ is minimal. From now on, let $(j_0,m_0,n_0)$ stand for this index. 

Then with $\xi_k$ satisfying (\ref{xik}) with $j=j_0$, $m=m_0$, $n=n_0$, and $\epsilon = h^\gamma$, we have
\begin{align}
|P^{(\infty)}_1(k) - P^{(\infty)}_2(k)| 
= (P^{(\infty)}_{1,j_{_0}m_{_0}n_{_0}}-P^{(\infty)}_{2, j_{_0} m_{_0} n_{_0}})h^{j_0\beta + m_0\gamma +n_0} + o(h^{j_0\beta+m_0\gamma + n_0}),
\end{align}
and so for $h$ sufficiently small, 
\begin{align}\label{lowerbound}
|P^{(\infty)}_1(k) - P^{(\infty)}_2(k)| \gtrsim h^{j_0\beta + m_0\gamma + n_0}.
\end{align}
By assumption, $P^{(\infty)}_1(k) = P^{(\infty)}_2(\ell)$ for some $\ell \in \mathbf{Z}^2$, and then by Lemma \ref{lemma}, we have $k = \ell$ for all $h$ sufficiently small. Thus $P^{(\infty)}_1(k) = P^{(\infty)}_2(k)$, which contradicts (\ref{lowerbound}).

Therefore, we have $P^{(\infty)}_{1,jmn} = P^{(\infty)}_{2,jmn}$ for all $j,m,n$, so the two QBNF's are equal. 
\end{proof}

\begin{remark}
We remark that the complete integrability assumptions were unnecessary in the proof of Theorem \ref{thm1}. In principle, one only needs that the asymptotic expansions given in Theorem 1.1 of \cite{HSV} are valid, as well as the Diophantine assumption, and so the main result may hold more generally. 
\end{remark}

\begin{remark}
We note as an addendum to the discussion in Section \ref{BNFuniq} that when an operator satisfies the hypotheses of Theorem \ref{thm1}, it implies uniqueness of the QBNF near the Diophantine torus $\Lambda$. Indeed, Theorem \ref{hsvthm} describes the eigenvalues of such an operator in a window $R_\beta\subseteq \mathbf{C}$ corresponding to $\Lambda$, and Theorem \ref{thm1} implies QBNF near $\Lambda$ can be (uniquely) recovered from the eigenvalues if we know the frequencies $a$. This implies that the QBNF is unique up to the choice of action-angle variables. Thus we have
\end{remark}
\begin{corollary}\label{coruniq}
If $P_\epsilon$ is an operator satisfying the hypotheses on the operators in Theorem \ref{thm1}, the QBNF of $P_\epsilon$ near the Diophantine torus $\Lambda$ is uniquely defined up to the choice of action-angle variables. 
\end{corollary}

\begin{remark}
In the proof of Theorem \ref{thm1}, we exploited the fact that Theorem \ref{hsvthm} applies for all $\epsilon$ in a range $h^K \leq \epsilon \leq h^\delta$ to assume that $\epsilon = h^\gamma$, for a favorable choice of $\gamma$. It is natural to consider situations when $\epsilon$ is, for example, a function of $h$, or possibly has a more general sort of degenerate relationship with $h$. For example:


\begin{enumerate}
\item The damped wave operator on a compact manifold may be studied as a non-selfadjoint perturbation of the selfadjoint operator $-h^2\Delta$, where $\Delta$ is the Laplace-Beltrami operator, and the strength of the non-selfadjoint perturbation is $\epsilon = h$ (see \cite{SjDWE}). Then, for instance, since $\epsilon^2 = \epsilon h = h^2$, it is meaningless to ask for the coefficients of these terms individually in some QBNF for the operator. 

\item If $P^{(\infty)}$ is an asymptotic expansion satisfying (formally)\begin{align}\label{rmultiple} P^{(\infty)}(\xi,\epsilon,h) = (\epsilon-h)(\epsilon^2-h)\tilde P^{(\infty)}(\xi,\epsilon,h), \end{align} for some asymptotic expansion  $\widetilde P^{(\infty)}$,  then when we restrict to $\epsilon\in \{h,h^\frac{1}{2}\}$,  $P^{(\infty)}$ represents the zero function, and a QBNF may only be uniquely-defined modulo expansions of this form.
\end{enumerate}

More generally, consider situations where $\epsilon,h$ satisfy $r(\epsilon,h) = 0$, where $r \in C^\infty(\mbox{neigh}(0,0))$ is not flat at (0,0), and $r(0,0) = 0$. Replacing $r$ by $h^{-n_0}r$ if necessary, we may assume without loss of generality that $\partial_\epsilon^k r(0,0)\neq 0$ for some $k \geq 1$.

If $\partial_\epsilon r(0,0) \neq 0$, then by the implicit function theorem we have locally $\epsilon = f(h)$ for some smooth function $f(h)$. Thus 
\begin{align*}
r(\epsilon,h) = c(\epsilon,h)(\epsilon-f(h)),
\end{align*}
where $c(0,0) \neq 0$. Then $r(\epsilon,h) = 0$ precisely when $\epsilon = f(h)$. 

In general, if we have $\partial_\epsilon^k r(0,0) = 0$, $0 \leq k \leq m-1$, and $\partial_\epsilon^m r(0,0) \neq 0$, then the Malgrange preparation theorem (cf. \cite{HormanderVolI}, Section 7.5) implies a factorization
\begin{align}\label{malgrange}
r(\epsilon,h) = c(\epsilon,h)(\epsilon^m + a_{m-1}(h)\epsilon^{m-1} + \ldots +a_0(h)),\quad (\epsilon,h) \in \mbox{neigh}((0,0)).
\end{align}
where $c$ and $a_j$, $0\leq j \leq m-1$, are smooth functions of $(\epsilon,h)$ and $h$, respectively, with $c(0,0)\neq 0$, $a_j(0) = 0$.
As $|c(\epsilon,h)|$ is larger than some fixed, positive constant in a neighborhood of (0,0), let us assume without loss of generality that
\begin{align*}
r(\epsilon,h) = \epsilon^m + a_{m-1}(h)\epsilon^{m-1} + \ldots + a_0(h).
\end{align*}
Then by Theorem A.III.I of \cite{GerardPoles}, the roots of the right hand side, considered as a polyomial in $\epsilon$, have formal asymptotic expansions in Puiseux series, i.e. powers of $h^{1/k}$, for some fixed $k\in \mathbf{N}$. Thus on the level of formal power series we have
\begin{align*}{}
\epsilon^m + a_{m-1}(h)\epsilon^{m-1} + \ldots + a_0(h) = \prod_{i = 1}^m (\epsilon-f^{(i)}(h^\frac{1}{k})),
\end{align*}
where $f^{(i)}(h^\frac{1}{k})$ represents a formal Puiseux series,
\begin{align*}
f^{(i)}(h^\frac{1}{k}) \sim \sum_{n=0}^\infty c^{(i)}_n h^\frac{n}{k}.
\end{align*}
Using a Borelian construction, we may find smooth functions $\tilde f^{(i)}(h)$, $1 \leq i \leq m$, defined when $h \geq 0$, with asymptotic expansion near $h = 0$ given by $f^{(i)}(h)$. Then 
\begin{align}\label{factors}
\epsilon^m + a_{m-1}(h)\epsilon^{m-1} + \ldots + a_0(h) = \prod_{i = 1}^m (\epsilon-\tilde f^{(i)}(h^\frac{1}{k}) + \mathcal{O}(h^\infty)).
\end{align}
We claim that $r(\epsilon,h) = \mathcal{O}(h^\infty) \iff \dist(\epsilon,\cup_{i=1}^m f^{(i)}(h^\frac{1}{k})) = \mathcal{O}(h^\infty)$. 

Indeed, ($\Leftarrow$) is clear from (\ref{factors}). For ($\Rightarrow$), if $\dist(\epsilon,\cup_{i=1}^m f^{(i)}(h^\frac{1}{k})) \neq \mathcal{O}(h^\infty)$, then, possibly after restricting to a sequence of values of $h$ tending to zero, we have that $|\epsilon - \tilde f^{(i)}(h^\frac{1}{k})| \geq \frac{1}{\mathcal{O}(1)} h^{N_i}$, with $N_i \in \mathbf{N}$, $1 \leq i \leq m$, hence
\begin{align*}
|r(\epsilon,h)| = \prod_{i=1}^m |\epsilon-\tilde f^{(i)}(h) + \mathcal{O}(h^\infty)| \geq \prod_{i=1}^m (\frac{1}{\mathcal{O}(1)}h^{N_i} - \mathcal{O}(h^\infty)) \geq \frac{1}{\mathcal{O}(1)} h^{N_1+ \ldots +N_m} \neq \mathcal{O}(h^\infty).
\end{align*}
Conversely, if $\epsilon$ is a smooth (real-valued) function of $h^\frac{1}{k}$, 
\begin{align}\label{single}
\epsilon=  f(h^\frac{1}{k}) \sim \sum_{n=0}^\infty c_n h^\frac{n}{k}, \quad c_n \in \mathbf{R},
\end{align}
then by a Borelian construction, we can find smooth functions $f^{(i)}$, $0 \leq i \leq k-1$ with asymptotics given by the Puiseux conjugates of (\ref{single}), i.e.
\begin{align*}
f^{(i)}(h^\frac{1}{k}) \sim \sum_{n=0}^\infty c_n (\zeta_k^i h^\frac{1}{k})^n,
\end{align*}
where $\zeta_k$ is a primitive $k$th root of unity. (We can take $f^{(0)} = f$.) Then one can check that
\begin{align}\label{singledegen}
r(\epsilon,h) = \prod_{i = 1}^m (\epsilon - f^{(i)}(h^\frac{1}{k})) 
\end{align}
is a smooth function near $(0,0)\in \mathbf{R}^2$, and $r(\epsilon,h) = \mathcal{O}(h^\infty)$ when $\epsilon = f(h)$. 

If $\epsilon \in \{ f_\mu(h^\frac{1}{k_\mu})\st 1\leq \mu \leq M\}$, where each $f_\mu$ is a smooth, real-valued function near 0, then letting $r(\epsilon,h)$ be the product of the corresponding expressions (\ref{singledegen}) for each $f_\mu$, we get a smooth function $r(\epsilon,h)$ such that $r(\epsilon,h) = \mathcal{O}(h^\infty)$ along the union of the curves $\epsilon = f_\mu(h^\frac{1}{k_\mu})$. 


To sum up the discussion, we see that there is a degenerate relationship between the parameters $\epsilon,h$ precisely when $\epsilon$ is within $\mathcal{O}(h^\infty)$ of a finite number of curves of the form $\epsilon = \tilde f^{(i)}(h^\frac{1}{k})$, where $\tilde f^{(i)}$ is a smooth, real-valued function near 0. 
\end{remark}

\bibliography{mybib}{}
\bibliographystyle{abbrv}

\end{document}